\documentclass[11pt]{article}

\usepackage[T1]{fontenc}
\usepackage[utf8]{inputenc}
\usepackage{lmodern}
\usepackage{microtype}
\usepackage[margin=1in]{geometry}
\usepackage{setspace}

\usepackage{amsmath,amssymb,amsthm,mathtools}
\numberwithin{equation}{section}
\newtheorem{theorem}{Theorem}

\newtheorem{proposition}{Proposition}
\newtheorem{corollary}[theorem]{Corollary}
\theoremstyle{definition}

\newtheorem{question}{Question}

\theoremstyle{remark}
\newtheorem{remark}[theorem]{Remark}

\usepackage{graphicx}
\graphicspath{{figures/}}
\usepackage{booktabs}
\usepackage{tabularx}
\usepackage{longtable}
\usepackage{array}
\usepackage{enumitem}
\usepackage{float}
\usepackage{caption}
\providecommand{\tightlist}{\setlength{\itemsep}{0pt}\setlength{\parskip}{0pt}}

\usepackage{xcolor}
\usepackage[backend=biber,style=numeric,sorting=nyt]{biblatex}
\usepackage[hidelinks]{hyperref}
\usepackage[nameinlink,noabbrev]{cleveref}
\crefname{question}{question}{questions}
\Crefname{question}{Question}{Questions}
\hypersetup{
  pdftitle={Exact Area-Range Minima in the Quantitative Monsky Problem for Five and Seven Triangles},
  pdfauthor={Muxi Li}
}

\title{Exact Area-Range Minima in the Quantitative Monsky Problem\\
for Five and Seven Triangles\\[0.4em]
\large Computer-Assisted Certification, a Nine-Triangle Upper Bound,\\
and Limits of the Single-Cap Zig-Zag Family}
\author{Muxi Li\\
\small\texttt{limuxi@ustc.edu.cn}}

\begin{document}
\maketitle
\begin{abstract}
For a dissection \(D\) of the unit square into \(n\) nondegenerate triangles, let
\[
R(D)=\max_i a_i-\min_i a_i,
\qquad
\Delta(n)=\inf_D R(D).
\]
We first note that this infimum is attained for every \(n\ge2\), so \(\Delta(n)\) is always a genuine minimum. We then determine the exact minima for \(n=5\) and \(n=7\), allowing general triangular dissections with T-junctions. For five triangles,
\[
\Delta(5)=\frac{5\sqrt5-11}{8},
\]
and equality forces three areas \((3-\sqrt5)/4\) and two areas \((3\sqrt5-5)/8\). For seven triangles,
\[
\Delta(7)=r_7=0.0002011756316409390820\ldots,
\]
where \(r_7\) is the unique root in \((0,1/4900)\) of
\[
864r^4+2160r^3-6060r^2+4972r-1.
\]
Moreover every minimizing seven-triangle dissection has exactly four areas
\[
H_7=\frac{1+3r_7}{7}
\]
and three areas
\[
L_7=\frac{1-4r_7}{7},
\]
although the minimizing geometry need not be unique.
The five-triangle proof uses a finite computer-certified classification followed by exact analytic inequalities. The seven-triangle result uses exact graph enumeration, a five-parameter multiaffine rational model permitting cyclic T-junction dependencies, integer interval certificates, and analytic treatment of the seven surviving types. No floating-point optimization enters the final proof chain.

We also study the first case beyond these exact results. Although \(n=3,5,7\) share a two-level single-cap zig-zag closure mechanism, that geometry is already noncompetitive at \(n=9\). A tilted-strip construction gives
\[
\Delta(9)\le r_{\mathrm{tilt}}
=0.0001273496861283553341\ldots,
\]
where \(r_{\mathrm{tilt}}\) is the unique root in \((0.000127,0.000128)\) of
\[
4864r^4-824r^3-18804r^2-7850r+1,
\]
and this value is the exact minimum within the tilted-strip topology. Conversely, every nine-triangle dissection in the complete single-cap two-rail zig-zag family, with arbitrary continuous areas and all \(2^8\) direction sequences, has \(R>1/3500\). Thus a global nine-triangle minimizer must leave that family. The exact value of \(\Delta(9)\) remains open.
\end{abstract}

\section{Introduction}

Monsky's theorem~\cite{Monsky1970} states that a square cannot be dissected into an odd number of equal-area triangles. The quantitative version asks how close an odd dissection can come to equal areas. Following Labbé, Rote, and Ziegler~\cite{LabbeRoteZiegler2020}, we measure the deviation by the range

\[
R(D)=\max_{i,j}|a_i-a_j|=\max_i a_i-\min_i a_i,
\]

and define \(\Delta(n)\) to be the infimum of this range over dissections of the unit square into \(n\) triangles. A \emph{dissection} is allowed to have T-junctions; it need not be a simplicial triangulation. An elementary compactness argument, recorded in \cref{prop:attainment}, shows that the infimum is attained for every \(n\ge2\). Consequently, Monsky's theorem already implies nonconstructively that \(\Delta(n)>0\) for every fixed odd \(n\); the quantitative lower bounds of Labbé--Rote--Ziegler provide explicit estimates and asymptotic information.

Labbé, Rote, and Ziegler~\cite{LabbeRoteZiegler2020} established general lower bounds from real algebraic geometry and superpolynomial upper bounds from explicit zig-zag constructions. Their computations enumerated dissection types with at most eight skeleton nodes and identified small-\(n\) configurations with low area deviation~\cite[Section~6]{LabbeRoteZiegler2020}. A complete global treatment of seven triangles also requires the nine-node layer. Here we determine two small odd cases exactly and then use the next case to test the geometric pattern suggested by them.

Our first result gives the exact five-triangle minimum.

\begin{theorem}[Five triangles]\label{thm:n5}
For every dissection \(D\) of the unit square into five nondegenerate triangles, including dissections with T-junctions,

\[
R(D)\ge R_5:=\frac{5\sqrt5-11}{8}.
\]

Equality is attainable. Moreover, equality forces the multiset of areas to be

\[
\{m,m,m,M,M\},\qquad
m=\frac{3-\sqrt5}{4},\qquad
M=\frac{3\sqrt5-5}{8}.
\]
\end{theorem}

Our second result gives the exact seven-triangle minimum.

\begin{theorem}[Seven triangles]\label{thm:n7}
Let

\[
Q_7(r)=864r^4+2160r^3-6060r^2+4972r-1.
\]

There is a unique root \(r_7\in(0,1/4900)\), and every dissection \(D\) of the unit square into seven nondegenerate triangles satisfies

\[
R(D)\ge r_7.
\]

Numerically,
\[
r_7=0.0002011756316409390820336285676\ldots.
\]
Define
\[
H_7=\frac{1+3r_7}{7},\qquad
L_7=\frac{1-4r_7}{7}.
\]
Then a seven-triangle dissection is minimizing if and only if its area multiset is
\[
\{H_7,H_7,H_7,H_7,L_7,L_7,L_7\}.
\]
In particular every minimizer has exactly two distinct areas, four high and three low. This statement does not assert uniqueness of the minimizing geometry.
\end{theorem}

The proofs are exact, but the role of computation is different in the two cases. For \(n=5\), the machine is used to certify a finite classification and symbolic area identities; the sharp lower bound is then derived analytically. For \(n=7\), exact computation is also used for the global elimination of a large but finite collection of combinatorial types. The proof is nevertheless organized so that the mathematical meaning of every computational certificate is explicit. Executable certificates, exact machine-readable outputs, and reproduction instructions are provided in the companion proof artifact described in Section~\ref{sec:repro}.

A further theme is what happens immediately beyond the exact cases. The values for \(n=3,5,7\) arise as small positive roots of closure polynomials in a common two-level zig-zag family. At \(n=9\), however, the failure is stronger than a mismatch with the standard \(n\mapsto n+2\) extension. We construct a tilted-strip nine-triangle topology whose exact family minimum is
\[
r_{\mathrm{tilt}}=0.0001273496861283553341\ldots,
\]
strictly better than the optimized straight-strip extension of the seven-triangle optimum. At the same time, we prove that the entire single-cap two-rail nine-triangle zig-zag family, with arbitrary continuous areas and arbitrary direction sequence, satisfies \(R>1/3500\). Thus the old geometric family itself is globally noncompetitive at \(n=9\). We do not conjecture that the tilted-strip value is \(\Delta(9)\); instead we isolate a weaker two-level question and describe the complete enumeration problem that would be needed to determine \(\Delta(9)\) exactly.

\section{Definitions and proof architecture}

Throughout, the square is the unit square. A \emph{triangular dissection} is a finite collection of nondegenerate triangles with pairwise disjoint interiors whose union is the square. T-junctions are allowed. A \emph{triangulation} is the special case in which two triangles meet only in a common full edge, a common vertex, or not at all.

Before introducing the combinatorial machinery, we record a compactness observation that justifies the word ``minimum'' throughout the paper.

\begin{proposition}[Attainment of the minimum]\label{prop:attainment}
For every integer \(n\ge2\), the infimum defining \(\Delta(n)\) is attained. In particular,
\[
\Delta(n)=\min_{D\in\mathcal D_n}R(D),
\]
where \(\mathcal D_n\) denotes the space of genuine \(n\)-triangle dissections of the unit square.
\end{proposition}

\begin{proof}
Let \(Q=[0,1]^2\). The genuine dissection space is obviously nonempty for every \(n\ge2\). Consider the set \(\mathcal K_n\subset(Q^3)^n\cong[0,1]^{6n}\) of ordered vertex arrays whose convex hulls
\[
T_1,\ldots,T_n\subseteq Q
\]
cover \(Q\) and have pairwise disjoint Euclidean interiors; degenerate triangles are allowed in \(\mathcal K_n\). Labels and vertex orderings are kept, so no quotient-space argument is needed.

The set \(\mathcal K_n\) is closed. Indeed, convergence of the vertices implies Hausdorff convergence of the corresponding convex hulls. The covering condition passes to the limit: for a fixed \(x\in Q\), from \(x\in\bigcup_iT_i^{(m)}\) one may pass to a subsequence on which the covering index is constant. Pairwise interior-disjointness also passes to the limit: if two limiting nondegenerate triangles had a common interior point, the barycentric coordinates of that point would be strictly positive in each triangle and would remain positive under sufficiently small perturbations of the vertices, forcing interior overlap for all sufficiently large \(m\). Thus \(\mathcal K_n\) is a closed subset of a compact cube and hence is compact.

The area of a triangle is a continuous function of its vertices. Therefore
\[
R(T_1,\ldots,T_n)
 =\max_i\operatorname{area}(T_i)-\min_i\operatorname{area}(T_i)
\]
is continuous on \(\mathcal K_n\), and hence attains a minimum there.

Suppose that a minimizer in \(\mathcal K_n\) contains a degenerate triangle. Let \(k<n\) be the number of positive-area triangles and let
\[
M=\max_i\operatorname{area}(T_i).
\]
Since at least one triangle has area zero, the range of this generalized dissection is exactly \(M\). Because the triangles cover \(Q\) and have pairwise disjoint interiors, the sum of their areas is \(1\); the degenerate pieces contribute zero. Hence the finite union of the \(k\) positive-area triangles is a closed subset of \(Q\) of area \(1\). A nonempty relatively open complement in \(Q\) would have positive area, a contradiction. Thus the positive-area triangles alone cover \(Q\).

Choose one positive-area triangle and subdivide it into \(n-k+1\) positive-area triangles, for instance by joining one vertex to \(n-k\) distinct interior points of the opposite side. Replacing the chosen triangle by these pieces produces a genuine \(n\)-triangle dissection. Every new area is positive and no area exceeds \(M\), so the new range is strictly smaller than \(M\). This contradicts minimality on \(\mathcal K_n\). Hence every minimizer in \(\mathcal K_n\) is nondegenerate and belongs to \(\mathcal D_n\), proving the claim.
\end{proof}

In view of \cref{prop:attainment}, we freely refer to \(\Delta(n)\) as the minimum area range. For odd \(n\), Monsky's theorem implies that this minimum is strictly positive, although it does not by itself give an effective lower bound.

The \emph{skeleton graph} has as nodes all triangle vertices and has an edge between consecutive nodes lying on the same side of a triangle. A node that lies in the relative interior of a side of another triangle is a side node or T-junction. The distinction between actual triangle corners and side nodes is essential: a proof of the global \(n=7\) theorem would be incomplete if it enumerated only ordinary triangulations.

The computer-assisted portions of the proof follow a common principle. The paper proves that every geometric dissection belongs to a finite list of combinatorial data; for each datum, the relevant geometric parameter space is embedded in a box; exact arithmetic certifies that either the box cannot contain a near-equal-area dissection or the type belongs to a short list treated analytically. Floating-point searches were useful during discovery, but no numerical optimizer is part of the proof.

\section{The exact five-triangle minimum}\label{sec:n5}

\subsection{A sharp construction}

Let

\[
\tau=\frac{3-\sqrt5}{2},
\qquad
A=(0,0),\ B=(1,0),\ C=(1,1),\ D=(0,1),
\]

and set

\[
E=(1,\tau),
\qquad
P=\left(\tau,\frac{1+\tau^2}{2}\right).
\]

The five triangles

\[
DAP,\quad DCP,\quad CEP,\quad AEP,\quad ABE
\]

form a dissection of the square. Their doubled areas are

\[
\tau,
\quad \frac{1-\tau^2}{2},
\quad (1-\tau)^2,
\quad \frac{1-\tau^2}{2},
\quad \tau.
\]

Since \((1-\tau)^2=\tau\), the actual areas are three copies of

\[
m=\frac{\tau}{2}=\frac{3-\sqrt5}{4}
\]

and two copies of

\[
M=\frac{1-\tau^2}{4}=\frac{3\sqrt5-5}{8}.
\]

Thus \(R(D)=M-m=R_5\).

\subsection{Finite reduction to 319 certified combinatorial classes}

The node formula of Labbé--Rote--Ziegler~\cite{LabbeRoteZiegler2020} implies that a five-triangle dissection of a square has six or seven skeleton nodes. The enumeration used here begins from all 3-connected simple planar graphs on seven and eight vertices, chooses the exterior apex, chooses the four square corners, assigns the three true triangle corners of each bounded face, and retains a deliberate over-enumeration of all geometrically possible dissections. After isomorphism reduction there are

\begin{itemize}
\tightlist
\item
  20 classes with six skeleton nodes; and
\item
  299 classes with seven skeleton nodes.
\end{itemize}

Every accepted type has exactly three continuous geometric parameters; compare the parameter count in Campbell--Brady--Nair~\cite{CampbellBradyNair2007}. The exact symbolic area functions of those parameters are then used to certify one of a finite list of area identities.

The classification is summarized in Table 1. The symbols \(Z\) and \(H\) denote, respectively, a forced zero area and a forced half-area relation. The remaining classes satisfy one of nine polynomial relations \(F_0,\ldots,F_8\) after a permutation of the five area variables.

\begin{table}[htbp]
\centering
\caption{Certified classification of the five-triangle combinatorial classes.}
\label{tab:n5-classification}
\small
\setlength{\tabcolsep}{4pt}
\begin{tabular}{@{}rrrrrrrrrrrrr@{}}
\toprule
$N$ & $Z$ & $H$ & $F_0$ & $F_1$ & $F_2$ & $F_3$ & $F_4$ & $F_5$ & $F_6$ & $F_7$ & $F_8$ & total \\
\midrule
6 & 3 & 15 & 1 & 1 & 0 & 0 & 0 & 0 & 0 & 0 & 0 & 20 \\
7 & 142 & 126 & 5 & 0 & 5 & 11 & 3 & 1 & 3 & 2 & 1 & 299 \\
\bottomrule
\end{tabular}
\end{table}

Here the four variables \(a,b,c,d\) label distinct triangles, but their area values may coincide; the fifth area is \(1-a-b-c-d\). The relations used in the classification are
\[
\begin{aligned}
F_0={}&4ab+2c+2d-1,\\
F_1={}&-4ad+4bc-2b-2c+1,\\
F_2={}&4ab-2a-2b-2c-2d+1,\\
F_3={}&4ab+4ac-2a-2b-2c-2d+1,\\
F_4={}&-4a^2-4ac-4ad+4a+4bc+2d-1,\\
F_5={}&4a^2+8abc+8ad-4a-4bc+4bd+4d^2-4d+1,\\
F_6={}&8a^2b-4a^2+8abc-8ab-4ac-4ad+4a+2b+2d-1,\\
F_7={}&8a^2b-4a^2+8abc-8ab-8ac+4a-4bc+2b\\
&\quad-4c^2-4cd+4c-1,\\
F_8={}&4ab(1-2c-2d)+4ac(1-2c)+4bd(1-2d)\\
&\quad-(1-2c)(1-2d).
\end{aligned}
\]

The machine check is symbolic: after substituting the rational area functions into a candidate relation, denominators are cleared and every coefficient of the numerator polynomial is checked to vanish. Numerical sampling is not used as a certificate.

\subsection{A narrow near-equality window}

If \(R\le R_5\), then in particular \(R<1/40\). Since the five areas have mean \(1/5\), every area lies in

\[
I=\left(\frac7{40},\frac9{40}\right).
\]

This narrow interval is enough to eliminate all certified relations except \(F_0\).

A \(Z\)-type cannot be a nondegenerate dissection. In an \(H\)-type, either one area is \(1/2\), which forces another area to be at most \(1/8\), or two areas sum to \(1/2\), which forces a gap of at least \(1/12\). Both bounds are much larger than \(R_5\).

For \(F_1,F_2,F_3,F_4,F_5,F_7,F_8\), write each variable as \(1/5+x_i\) with \(|x_i|<1/40\). Expanding the corresponding polynomial about the equal-area point and bounding every nonconstant monomial by exact rational arithmetic shows that the polynomial cannot vanish on \(I^4\). The remaining relation \(F_6\) is excluded by a direct derivative monotonicity argument on the closed box \([7/40,9/40]^4\).

Thus every dissection with \(R\le R_5\), including every possible equality case, must satisfy \(F_0=0\).

\subsection{\texorpdfstring{The sharp \(F_0\) inequality}{3.4. The sharp F\_0 inequality}}\label{the-sharp-f_0-inequality}

After renaming the five areas as \(u,v,x,y,w\), the critical relation is

\[
x+y+2uv=\frac12,
\qquad
u+v+x+y+w=1.
\]

Balancing \(x\) and \(y\) cannot increase the range, so we may replace them by their common average

\[
s=\frac14-uv.
\]

Next set \(t=\sqrt{uv}\). Replacing \((u,v)\) by \((t,t)\) and adjusting the fifth area to

\[
\widetilde w=\frac12-2t+2t^2
\]

again does not increase the range in the relevant interval. Indeed, \(t\) lies between \(u\) and \(v\), while
\[
w\le\widetilde w<s,
\qquad
s-\widetilde w=-\frac14+2t-3t^2\ge\frac{13}{1600}
\]
on \([7/40,9/40]\). Thus neither balancing step moves a value outside the previous area interval. The problem therefore reduces to minimizing

\[
r(t)=\operatorname{range}\left\{t,\frac14-t^2,\frac12-2t+2t^2\right\}
\]

for \(t\in(7/40,9/40)\).

Let

\[
m=\frac{3-\sqrt5}{4},
\qquad
\sigma=\frac{\sqrt2-1}{2}.
\]

The ordering of the three functions changes only at \(m\) and \(\sigma\). On \([7/40,m]\) the range is strictly decreasing; on \([m,\sigma]\) it is strictly increasing; and on \([\sigma,9/40]\) it remains strictly increasing. Hence the unique minimum occurs at \(t=m\), and

\[
r(m)=\left(\frac14-m^2\right)-m=\frac{5\sqrt5-11}{8}.
\]

At equality, the unique one-variable minimum forces \(t=m\), so \(s=M\) and \(\widetilde w=m\). If \(u\ne v\), then \(u+v>2m\) and hence \(w<\widetilde w=m\), whereas \(\max(x,y)\ge M\); this would give range greater than \(R_5\). Therefore \(u=v=w=m\). Finally, \(x+y=2M\) forces \(x=y=M\), since otherwise one of them would exceed \(M\). This proves the equality multiset \(\{m,m,m,M,M\}\).

\section{The seven-triangle zig-zag mechanism}\label{sec:n7-zigzag}

\subsection{Closure of a cap-plus-zig-zag dissection}

Let the square corners be

\[
P=(0,0),\quad Q=(1,0),\quad T=(1,1),\quad S=(0,1).
\]

Choose a cap of area \(c\) at the upper right corner and set

\[
R_c=(1,1-2c),\qquad
C=\frac1{4c}.
\]

The remaining trapezoid is cut into six triangles by a zig-zag. Let their areas be \(a_1,\ldots,a_6\) and let

\[
A_i=a_1+\cdots+a_i,
\qquad
\rho_i=\frac{C-A_i}{C-A_{i-1}}.
\]

A sign \(s_i\in\{-1,+1\}\) records whether the \(i\)th step advances along the lower or upper rail. The recursive coordinate construction closes exactly when

\[
\Phi_s(c;a_1,\ldots,a_6):=\prod_{i=1}^{6}\rho_i^{s_i}=1.
\]

This equation is both necessary and sufficient for geometric closure provided \(0<c<1/2\) and the areas are positive and sum to one.

\subsection{The optimal sign pattern and the quartic}

For

\[
s=(-,+,-,+,+,-),
\]

the closure condition simplifies to

\[
q_2q_5=(1-2c)q_1q_3,
\qquad
q_j=C-A_j.
\]

Set

\[
H=\frac{1+3r}{7},
\qquad
L=\frac{1-4r}{7},
\]

and choose

\[
(c;a_1,\ldots,a_6)=(H;H,L,H,L,L,H).
\]

Substitution gives the exact identity

\[
q_2q_5-(1-2H)q_1q_3
=
\frac{Q_7(r)}{2744(3r+1)}.
\]

The polynomial \(Q_7\) is strictly increasing on \([0,1/4900]\), with \(Q_7(0)<0<Q_7(1/4900)\). Thus it has a unique root \(r_7\) there. A rational isolating interval is

\[
\frac{201175631640939}{10^{18}}
<r_7<
\frac{201175631640940}{10^{18}}.
\]

At \(r=r_7\), these values are precisely \(H_7\) and \(L_7\) from \cref{thm:n7}:
\[
H_7=0.14294336098498897389\ldots,
\qquad
L_7=0.14274218535334803481\ldots.
\]

\begin{figure}[htbp]
\centering
\includegraphics[width=0.72\textwidth]{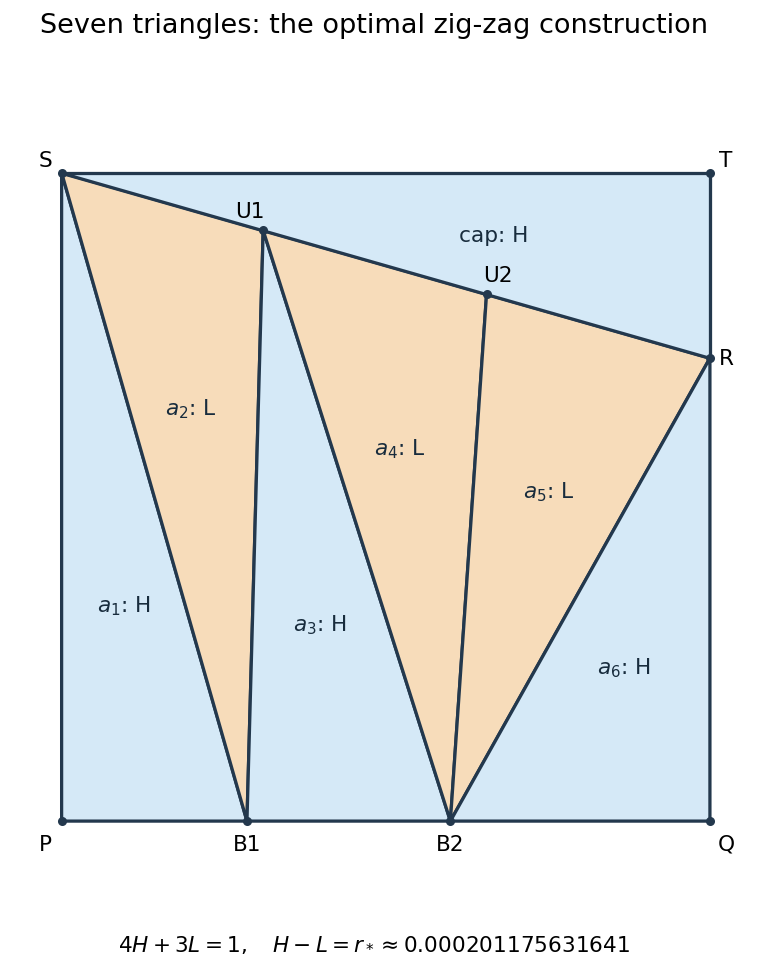}
\caption{An optimal seven-triangle zig-zag dissection. The four high-area triangles have area $H$ and the three low-area triangles have area $L$.}
\label{fig:n7-zigzag}
\end{figure}

\subsection{Continuous minimax within the full sign family}

The two-level pattern above is not assumed in advance. For a fixed target range \(r\), every area vector with sum one and range at most \(r\) can be written as

\[
x_i=\frac17+r(z_i-\bar z),
\qquad
0\le z_i\le1,
\qquad
\bar z=\frac17\sum_{j=0}^{6}z_j.
\]

For the sign pattern \((-,+,-,+,+,-)\), use the logarithmic closure function
\[
G=2\log q_2+2\log q_5-2\log q_1-2\log q_3-2\log(1-2c),
\]
which equals \(\log\Phi_s\) on the area-sum hyperplane. Write \(g_i=\partial_iG\) and \(\bar g=\frac17\sum_i g_i\). Exact rational interval bounds on the partial derivatives show that the centered gradient has sign pattern

\[
\mathrm{sgn}(g_i-\bar g)=(+1,+1,-1,+1,-1,-1,+1)
\]

throughout a box that contains every candidate with \(r\le1/4900\). Consequently \(G\) is strictly monotone in each cube coordinate and is maximized at the unique corner corresponding to

\[
(H;H,L,H,L,L,H).
\]

It follows that no vector of range \(r<r_7\) can close, while at \(r=r_7\) this corner does close.

The remaining 62 sign patterns are eliminated uniformly at the coarser threshold \(R\le1/4000\) by exact rational comparison with the equal-area closure values. The overall sign reversal gives the same closure equation. Hence the entire generalized six-step zig-zag family has minimum range \(r_7\).

\section{Global reduction for seven triangles}\label{sec:n7-global}

The main difficulty is to show that no other seven-triangle dissection has smaller range.

Let \(N\) be the number of skeleton nodes. Let \(b\) be the number of non-corner boundary nodes, \(t\) the number of interior side nodes (T-junctions), and \(j\) the number of ordinary interior nodes. An angle count gives

\[
N\in\{7,8,9\},
\qquad
b+t=2N-13,
\qquad
j=9-N.
\]

This divides the global proof into finitely many structural layers.

A useful coarse pruning principle is the following. If a line or subregion separates \(k\) triangles of total area \(A\) from the remaining \(7-k\), then

\[
R\ge\left|\frac Ak-\frac{1-A}{7-k}\right|
=
\frac{|7A-k|}{k(7-k)}.
\]

In particular, a complete square diagonal forces \(R\ge1/24\) and can be discarded far above the target scale.

The ordinary triangulations are exhausted according to the number of interior vertices. Exact integer interval certificates show that every ordinary seven-triangle triangulation has \(R>1/1000\). The same bound holds for all seven-node dissections, including those with one interior T-junction. The remaining eight- and nine-node layers require a more flexible model.

\section{A unified five-parameter model with cyclic T-junctions}\label{sec:five-param}

A striking simplification is that every remaining combinatorial type has exactly five continuous parameters. Merge the collinear atomic edges through each T-junction into maximal support chains. Non-T nodes are called anchors. Each non-corner boundary anchor contributes one parameter, each ordinary interior anchor contributes two coordinates, and each T-junction contributes one affine position parameter. Thus

\[
b+2j+t=5.
\]

Cyclic dependencies among T-junctions cannot be ignored. Let \(X\) be the matrix of unknown homogeneous coordinates of the T-nodes. Their affine relations can be written

\[
X=MX+B,
\qquad
A=I-M,
\qquad
D=\det A.
\]

For a valid interior parameter point, every T-node eventually reaches an anchor along the dependency graph. This implies \(\rho(M)<1\), hence \(A^{-1}\) exists and is nonnegative and \(D>0\).

For each triangular face, a Schur-complement determinant gives an exact numerator \(P_i\) such that

\[
\frac{P_i}{D}=2a_i^{\mathrm{signed}}.
\]

Crucially, \(D,P_1,\ldots,P_7\) are multiaffine functions of the same five parameters. At a box corner one may have \(D=0\), but then \(P_i=0\) as well. Multiaffine interpolation shows that on every parameter box the ratio \(P_i/D\) lies between the ratios at the nonsingular corners. This provides a rigorous interval enclosure without requiring a uniform positive lower bound for \(D\).

This is the key device that permits exact certification even when the T-junction support relation contains cycles.

\section{Enumeration, interval certification, and the seven survivors}\label{sec:certification}

\subsection{Complete graph enumeration}

After adjoining an exterior apex to the boundary cycle, every realizable skeleton becomes a simple 3-connected planar graph. Any nontriangular face can be triangulated by adding exactly \(t\) edges. Therefore every target skeleton is a subgraph of a maximal planar graph on the same vertex set.

The implementation enumerates maximal planar graphs by edge flips (using the fixed-vertex flip-connectivity framework; see, e.g.,~\cite{BurtonDattaSpreer2022}), uses uniqueness of the rotation system for 3-connected planar graphs, removes the admissible edges, assigns the straight-angle vertices of nontriangular faces, and finally chooses the four square corners among the boundary nodes. Canonical encodings eliminate isomorphic duplicates.

For the remaining eight- and nine-node layers, the marked counts are as follows.

\begin{table}[htbp]
\centering
\caption{Marked assignments in the remaining eight- and nine-node layers.}
\label{tab:n7-marked-counts}
\begin{tabular}{@{}rrr@{}}
\toprule
$N$ & boundary nodes $B$ & marked assignments \\
\midrule
8 & 4 & 1077 \\
8 & 5 & 2490 \\
8 & 6 & 2415 \\
9 & 4 & 1624 \\
9 & 5 & 6515 \\
9 & 6 & 12645 \\
9 & 7 & 14630 \\
9 & 8 & 10080 \\
\bottomrule
\end{tabular}
\end{table}

The enumerator intentionally produces a safe superset: some combinatorial assignments need not be geometrically realizable. This is harmless because the subsequent elimination is performed on the larger parameter space.

\subsection{Exact interval certificates}

Assume \(R\le1/4000\). Then every area lies in

\[
\left[\frac17-\frac1{4000},\frac17+\frac1{4000}\right].
\]

For each five-dimensional parameter box, the 32 corner values of \(P_i\) and \(D\) are computed exactly. If the entire ratio interval for any face misses the target area window, the box is discarded; otherwise the longest coordinate interval is bisected. All comparisons are integer cross-multiplications.

The general ratio engine processes 14,989 combinatorial cases and 361,046 boxes. An independent arbitrary-precision Python implementation reproduces the status, box count, and surviving boxes of the C++ implementation case by case. The C++ implementation uses signed 128-bit integers; a priori bounds keep all intermediate values below \(2^{105}\).

For the nine-node types, 45,494 marked assignments are partitioned into diagonal cases, forced nonpositive-area cases, directly recognized zig-zags, interval-eliminated cases, and seven survivors. The interval engine eliminates 14,270 of the 14,277 cases that actually enter this stage.

The seven fixed survivor indices are

\[
(6,2791),(6,2891),(6,3590),(6,3752),(6,3829),(7,2108),(7,2166),
\]

where the first coordinate is the boundary-node count \(B\) and the second is the zero-based case index in the corresponding certificate file.

\subsection{Four survivors by area-preserving recutting}

Four survivor types admit an exact local recutting. If two triangles partition a larger triangle by a segment through a point on one side, the same two areas can be reproduced by cutting from another vertex to a uniquely chosen point on another side. In the four relevant types this operation removes the exceptional T-junction and converts the dissection into the already solved generalized zig-zag family without changing the multiset of seven areas. Hence these four types satisfy \(R\ge r_7\).

\subsection{Three survivors by a common closure polynomial}

The remaining three types, with indices 3590, 3752, and 3829, have their seven areas relabeled as \(a,b,c,d,e,f,g\). All three satisfy the exact closure relation

\[
F(a,b,c,d,e,f,g)=0,
\]

where

\[
F=\bigl((1-2a)(1-2e)-2b\bigr)\bigl(1-4e(a+d)\bigr)
-2c(1-2e)(1-4ae).
\]

On the complete area box corresponding to range at most \(1/4000\), exact rational interval bounds give a fixed sign for every centered partial derivative \(\partial_iF-\overline{\partial F}\). Therefore, after the same cube parameterization of the range constraint used in the zig-zag proof, \(F\) is strictly monotone in each cube coordinate. Its unique minimum occurs at

\[
(a,b,c,d,e,f,g)=(H,H,H,L,H,L,L).
\]

At this point one has the exact identity

\[
F(H,H,H,L,H,L,L)=-\frac{Q_7(r)}{2401}.
\]

For \(r<r_7\), \(Q_7(r)<0\), so \(F>0\) throughout the entire admissible cube, contradicting the geometric closure equation \(F=0\). Equality constructions are obtained explicitly when \(r=r_7\).

\subsection{Equality rigidity: four high areas and three low areas}\label{subsec:n7-rigidity}

The same strict centered-gradient certificates also determine the area multiset at equality; no new enumeration or interval search is required. For \(r>0\), let
\[
\mathcal P_r=\left\{x\in\mathbb R^7:\ \sum_i x_i=1,\ \max_i x_i-\min_i x_i\le r\right\}.
\]
Every point of \(\mathcal P_r\) has a representation
\[
x_i=\frac17+r(z_i-\bar z),\qquad z\in[0,1]^7,
\qquad \bar z=\frac17\sum_i z_i.
\]
If a differentiable function \(U\) has a fixed strict sign for each centered partial derivative
\[
\partial_iU-\frac17\sum_j\partial_jU
\]
throughout \(\mathcal P_r\), then \(U\circ x\) is strictly monotone in every cube coordinate \(z_i\).  Its relevant extremum is therefore attained at the unique cube vertex selected by those signs.  Although this cube parameterization is not injective in general, the strict coordinate monotonicity forces every extremizing preimage to use the prescribed endpoints, so the extremizing area vector is unique.  When that vertex has four coordinates equal to one and three equal to zero, its area multiset is four copies of
\[
H(r)=\frac{1+3r}{7}
\]
and three copies of
\[
L(r)=\frac{1-4r}{7}.
\]

Choose the representative \((-,+,-,+,+,-)\) of the two surviving zig-zag sign sequences. Overall sign reversal negates the logarithmic closure function and leaves its zero set unchanged. For the chosen representative, the exact derivative certificate gives the strict centered-gradient sign pattern
\[
(+,+,-,+,-,-,+).
\]
Thus the logarithmic closure function \(G\) has a unique maximum on \(\mathcal P_{r_7}\), namely
\[
(H_7;H_7,L_7,H_7,L_7,L_7,H_7),
\]
and the quartic identity shows that this maximum is exactly the closing value \(G=0\). Hence every minimizing zig-zag dissection has four areas \(H_7\) and three areas \(L_7\).

The four recutting survivors preserve the full area multiset while converting the dissection to a solved zig-zag type, so they inherit the same rigidity. For the final three survivors, the exact centered-gradient signs of the common closure polynomial are
\[
(-,-,-,+,-,+,+).
\]
The unique minimum on \(\mathcal P_{r_7}\) is therefore
\[
(H_7,H_7,H_7,L_7,H_7,L_7,L_7),
\]
and the identity \(F=-Q_7(r_7)/2401=0\) forces equality to occur only there. All remaining combinatorial types were excluded at the strictly larger threshold \(1/4000\) (or stronger). Consequently
\[
R(D)=r_7
\quad\Longleftrightarrow\quad
\{a_1,\ldots,a_7\}_{\mathrm{multi}}
=\{H_7,H_7,H_7,H_7,L_7,L_7,L_7\}_{\mathrm{multi}},
\]
which proves the equality statement in \cref{thm:n7}.  This is rigidity of the area multiset only; distinct minimizing geometries are not ruled out.

Combining all layers proves \cref{thm:n7}.

\section{\texorpdfstring{A common closure-polynomial pattern for \(n=3,5,7\)}{A common closure-polynomial pattern for n=3,5,7}}\label{a-common-closure-polynomial-pattern-for-n357}

The exact values for three, five, and seven triangles can be placed in a single direct two-level zig-zag framework.

Let \(n=2m+1\) and let \(s=(s_1,\ldots,s_{2m})\) be a balanced sign sequence with \(s_i\in\{\pm1\}\) and \(\sum s_i=0\). Restrict all triangle areas to two values \(H>L\) with \(H-L=r\). We use the convention

\[
s_i=-1\iff a_i=H,
\qquad
s_i=+1\iff a_i=L,
\]

while the cap area \(c\) is chosen to be either \(H\) or \(L\). If there are \(h\) copies of \(H\) and \(\ell\) copies of \(L\) among all \(n\) triangles, then

\[
H=\frac{1+\ell r}{n},
\qquad
L=\frac{1-hr}{n}.
\]

The closure equation is

\[
\Phi_{n,s,c}(r)=
\prod_{i=1}^{n-1}
\left(\frac{C-A_i}{C-A_{i-1}}\right)^{s_i}=1,
\qquad
C=\frac1{4c}.
\]

After clearing denominators and factoring the numerator of \(\Phi-1\), the geometrically relevant small positive root is carried by a distinguished factor.

For the optimal small cases one obtains

\[
P_3(r)=4r-1,
\]

\[
P_5(r)=16r^2+44r-1,
\]

and

\[
P_7(r)=864r^4+2160r^3-6060r^2+4972r-1.
\]

Thus the exact values \(\Delta(3),\Delta(5),\Delta(7)\) are all selected by the same closure mechanism, even though the global proofs for \(n=5\) and \(n=7\) require substantially more than the direct zig-zag construction.

\begin{remark}[The direct \(n=9\) continuation]
The same algebraic closure procedure continues at \(n=9\). One direct two-level endpoint candidate occurs for
\[
s=-++-+--+,
\qquad c=L,
\]
and has distinguished factor
\[
P_9(r)=40960r^5+91648r^4+245632r^3+276560r^2+150308r-49,
\]
with small positive root
\[
0.000325801923513589\ldots.
\]
This remains a genuine member of the closure-polynomial family, but \cref{prop:n9-zigzag-exclusion} below shows that the entire single-cap two-rail geometry is noncompetitive for the global nine-triangle problem. Thus the direct \(P_9\) factor has structural rather than global-optimality significance.
\end{remark}

\section{Beyond seven triangles: a nine-triangle upper bound and the breakdown of the single-cap zig-zag paradigm}\label{sec:beyond-seven}

The exact cases \(n=3,5,7\) share a common two-level zig-zag closure mechanism. It is therefore natural to ask whether the same geometry remains competitive for larger odd \(n\). The answer is already negative for \(n=9\). The failure is stronger than the failure of any particular sign sequence or two-level ansatz: a new tilted-strip construction beats every member of the complete single-cap two-rail family.

\subsection{Inherited constructions as a baseline}

The extension construction recorded as Lemma~7.1 of Labbé--Rote--Ziegler~\cite{LabbeRoteZiegler2020} gives
\[
\Delta(n+2)\le \frac{n}{n+2}\Delta(n).
\]
Consequently,
\[
\Delta(9)\le \frac79r_7
=0.0001564699357207303971\ldots.
\]
For the present purpose one can optimize the elementary strip extension slightly by keeping track of the largest area in the seed dissection.

\begin{proposition}[Optimized straight-strip extension]\label{prop:straight-strip}
Suppose a unit-square dissection has minimum area \(L\), maximum area \(H\), and range \(r=H-L\). Attach a rectangular strip along one side and divide it into two congruent-area triangles of pre-normalization area \(a>0\). After the affine normalization back to a unit square, the range is
\[
\frac{\max(H,a)-\min(L,a)}{1+2a}.
\]
For the fixed seed dissection this expression is minimized at \(a=H\), giving
\[
R_{n+2}=\frac{r}{1+2H}.
\]
More generally, adding \(2m\) triangles of pre-normalization area \(H\) gives \(r/(1+2mH)\).
\end{proposition}

\begin{proof}
Before renormalization the old areas remain in \([L,H]\) and the two new areas are both \(a\), while the total area is \(1+2a\). This gives the displayed range after uniform area rescaling. For \(a\le H\) the expression decreases with \(a\), and for \(a\ge H\) it increases, so the minimum occurs at \(a=H\).
\end{proof}

By the equality rigidity in \cref{thm:n7}, every seven-triangle minimizer has maximum area \(H_7=(1+3r_7)/7\). Therefore
\[
\Delta(9)\le \frac{7r_7}{9+6r_7}
=0.0001564489532427382827\ldots.
\]
This small improvement is conceptually useful: inherited upper bounds depend not only on the seed range but also on the area profile of the seed.

\subsection{A tilted-strip nine-triangle construction}

We next deform the straight strip. Let the square have vertices
\[
P=(0,0),\quad Q=(1,0),\quad T=(1,1),\quad S=(0,1),
\]
and set
\[
H=\frac{1+4r}{9},\qquad
L=\frac{1-5r}{9},\qquad
c=H,
\]
\[
e=1-2H,\qquad f=1-2L,
\qquad E=(e,0),\qquad F=(f,1).
\]
Define
\[
Y=1-\frac{2c}{f},
\qquad
R=(e+(f-e)Y,Y).
\]
The right-hand quadrilateral \(EQTF\) is divided by the diagonal \(ET\) into two triangles of areas \(H\) and \(L\); the upper cap \(SFR\) has area \(H\). The remaining quadrilateral \(PERS\) is cut by the six-step direction sequence
\[
(-,+,-,+,+,-)
\]
into triangles of areas
\[
(H,L,H,L,L,H).
\]
The essential feature is that
\[
f-e=2r>0,
\]
so the interface \(EF\) is slightly tilted relative to the vertical sides of the square; see \cref{fig:n9-tilted}.

\begin{figure}[htbp]
\centering
\includegraphics[width=0.76\textwidth]{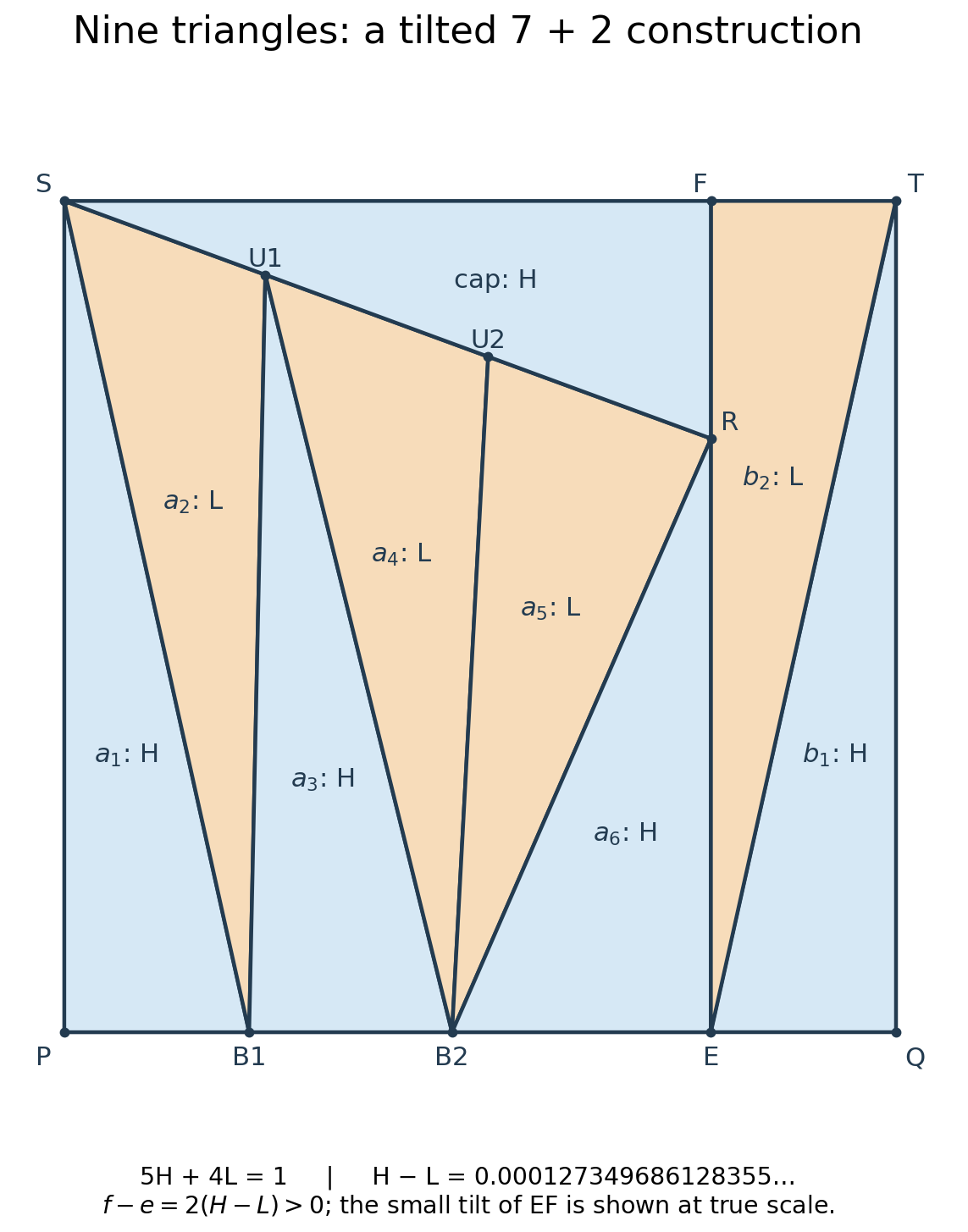}
\caption{The tilted \(7+2\) construction at \(r=r_{\mathrm{tilt}}\).  The nine areas are five copies of \(H\) and four copies of \(L\).  The interface \(EF\) has horizontal displacement \(f-e=2(H-L)>0\); the small tilt is shown at true scale.}
\label{fig:n9-tilted}
\end{figure}

Put
\[
D=e-f+\frac{f^2}{2c},
\qquad C=\frac D2,
\qquad A_i=\sum_{j=1}^i a_j,
\qquad q_i=C-A_i.
\]
The line through \(S\) and \(R\) meets the extension of the bottom side at \(O=(D,0)\). With the same multiplicative two-rail recursion as in the seven-triangle construction, the endpoint condition is equivalent to
\[
\prod_{i=1}^{6}\rho_i^{s_i}=\frac{2cD}{f^2},
\qquad \rho_i=\frac{q_i}{q_{i-1}}.
\]
For the displayed direction sequence this reduces, without squaring, to
\[
\mathcal F:=f q_2q_5-(f-2c)q_1q_3=0.
\]
Substituting the two-level areas gives the exact identity
\[
\mathcal F=-\frac{p(r)}{5832(4r+1)},
\]
where
\[
p(r)=4864r^4-824r^3-18804r^2-7850r+1.
\]

\begin{proposition}[Tilted-strip minimax]\label{prop:n9-tilted}
Let \(r_{\mathrm{tilt}}\) be the unique root of
\[
p(r)=4864r^4-824r^3-18804r^2-7850r+1
\]
in \((0.000127,0.000128)\). Then
\[
\Delta(9)\le r_{\mathrm{tilt}}
=0.000127349686128355334123703774\ldots.
\]
Moreover, within the fixed tilted-strip topology described above, the minimum possible area range is exactly \(r_{\mathrm{tilt}}\).
\end{proposition}

\begin{proof}
The root isolation interval is certified by exact rational arithmetic, and the geometric inequalities needed for the construction---in particular \(0<e,f<1\), \(f>2c\), \(D>e\), and \(q_i>0\)---hold uniformly on that interval. Hence the multiplicative recursion produces a legal noncrossing dissection when \(p(r)=0\).

For the lower bound within this topology, write the nine areas as
\[
x=(c,a_1,\ldots,a_6,b_1,b_2),
\qquad e=1-2b_1,
\qquad f=1-2b_2,
\]
and set \(\mathcal P(x)=8c\mathcal F(x)\). Every area vector of total sum one and range at most \(r\) can be parameterized by
\[
x_i=\frac19+r(z_i-\bar z),
\qquad 0\le z_i\le1.
\]
On the box
\[
\left[\frac19-\frac1{6300},\frac19+\frac1{6300}\right]^9
\]
exact rational derivative bounds give
\[
\operatorname{sgn}\!\left(9\partial_i\mathcal P-\sum_j\partial_j\mathcal P\right)
=(+,+,-,+,-,-,+,+,-).
\]
Thus \(\mathcal P\) is coordinatewise monotone on the range cube, and its maximum occurs at
\[
(c;a_1,\ldots,a_6;b_1,b_2)
=(H;H,L,H,L,L,H;H,L).
\]
At that point
\[
\mathcal P=-\frac{p(r)}{6561}.
\]
Since \(p\) is strictly decreasing on \([0,1/6300]\), with \(p(0)=1\) and unique zero \(r_{\mathrm{tilt}}\), no closure is possible for \(r<r_{\mathrm{tilt}}\). Equality is realized by the construction above.  The verifier \path{proof/n9/tilted_strip_n9_verifier.py} checks the quartic identities, the rational root bracket, the recurrence identities, the geometric inequalities, and all nine projected derivative signs using exact arithmetic. Its decimal coordinates and numerical triangle-area checks are auxiliary displays and do not establish the theorem.
\end{proof}

The proposition is deliberately a fixed-topology theorem. No complete enumeration of all nine-triangle skeletons is used here, so it does not identify \(\Delta(9)\).

\subsection{The complete single-cap two-rail family is noncompetitive}

The failure of the older zig-zag picture is considerably stronger than the fact that the direct two-level quintic root is too large.

\begin{proposition}[Exclusion of the single-cap two-rail family]\label{prop:n9-zigzag-exclusion}
Every nine-triangle dissection consisting of one cap having an entire side of the square as a side and an eight-step two-rail zig-zag in the complementary trapezoid satisfies
\[
R>\frac1{3500}.
\]
The statement allows arbitrary positive triangle areas and all \(2^8=256\) direction sequences; no balance or two-level assumption is imposed.
\end{proposition}

The certificate works on the complete area box corresponding to \(R\le1/3500\). For each direction sequence it bounds the centered gradient of the logarithmic closure function and integrates from the equal-area point. Elementary inequalities for \(\log p\) reduce the final comparisons to exact rational arithmetic. Overall sign reversal reduces the implementation to 128 representatives, and the weakest exclusion margin remains strictly positive.

Since
\[
r_{\mathrm{tilt}}<\frac1{3500},
\]
\cref{prop:n9-tilted,prop:n9-zigzag-exclusion} imply the following consequence.

\begin{corollary}\label{cor:n9-outside-single-cap}
Every range-minimizing nine-triangle dissection lies outside the single-cap two-rail family.
\end{corollary}

The obstruction is therefore geometric, not merely algebraic: changing the sign sequence or releasing the two-level restriction cannot repair the old topology.

\subsection{A weaker structural question}

The data suggest retaining a much weaker question rather than a new global zig-zag conjecture.

\begin{question}[Two-level optimality]\label{q:two-level}
For every odd \(n\ge3\), does there exist a minimizing dissection whose triangle areas take only two distinct values?
\end{question}

The exact minimizers for \(n=3,5,7\) are compatible with \cref{q:two-level}; for \(n=5\) and \(n=7\), the equality statements above show more strongly that every minimizer has the asserted two-level area multiset. The optimal point within the tilted-strip nine-triangle topology is again two-level. The latter is not known to be globally optimal, so it is not a fourth verified case.

There is a local variational reason why two-level patterns arise repeatedly. Suppose that near a legal configuration the area image is a smooth codimension-one patch
\[
\sum_{i=1}^n a_i=1,
\qquad F(a_1,\ldots,a_n)=0,
\]
and that the area map has the expected rank \(n-2\). At a local range minimizer, any intermediate coordinate \(L<a_i<H\) must satisfy
\[
\partial_iF=\frac1n\sum_j\partial_jF.
\]
Hence, if all centered normal components are nonzero, every area is forced to one of the two endpoints. This observation underlies the centered-gradient minimax arguments above. It is not a general theorem: singular points, rank drops, parameter-space boundaries, or vanishing centered components can support more complicated behavior.

\subsection{\texorpdfstring{Toward the exact value of \(\Delta(9)\)}{Toward the exact value of Delta(9)}}

The new upper bound gives a useful threshold, but an exact theorem requires a genuinely global classification. If \(b\) denotes the number of non-corner boundary nodes, \(t\) the number of internal straight/side nodes, and \(j\) the number of ordinary internal nodes, the nine-triangle node count gives
\[
b+t+2j=7,
\qquad
N=11-j,
\qquad
j\in\{0,1,2,3\}.
\]
Thus a complete proof must cover skeleton sizes \(N=8,9,10,11\); extending a seven-triangle configuration by two pieces cannot substitute for a complete enumeration.

A natural next program is therefore to enumerate all skeletons, straight-node assignments, and square-corner markings in these four node layers; verify a complete parameter representation, including cyclic T-junction dependencies and singular denominators; and then use an exact rational threshold slightly above \(r_{\mathrm{tilt}}\), for example \(1/7800\), to eliminate most types. Any survivors should be explored numerically without imposing a two-level ansatz, then converted into exact area relations, centered-gradient certificates, or elimination identities. A better construction would simply reset the comparison threshold and repeat the process.

For larger odd \(n\), upper-bound searches should distinguish three mechanisms: genuinely new primitive topologies, direct extensions inherited from smaller seeds, and geometric deformations of inherited constructions. The tilted strip belongs to the third class. This distinction is more robust than an envelope built only from direct single-cap zig-zags.

\begin{remark}[A similar obstruction at \(n=13\)]
The optional certificate \path{proof/n9/zigzag13_optional_exclusion.py} gives further evidence that the breakdown at \(n=9\) is not isolated. A legal eleven-triangle single-cap construction has range below \(8.21\times10^{-6}\), so the standard extension gives
\[
\Delta(13)<\frac{11}{13}\,8.21\times10^{-6}<\frac1{140000}.
\]
The same certificate excludes all 4096 single-cap thirteen-triangle direction sequences, with arbitrary continuous areas, below \(1/140000\). This does not determine \(\Delta(11)\) or \(\Delta(13)\); it only shows that the direct single-cap family can again be beaten by inheritance.
\end{remark}

\section{Computational certification and proof artifacts}\label{sec:repro}

The companion proof repository contains the executable certificates and machine-readable data used by the computer-assisted parts of the argument. The mathematical reductions specify what those programs certify: completeness of the combinatorial enumeration; exact area identities and interval enclosures; elimination or recording of every interval case; and analytic treatment of every survivor. Exploratory optimization and decimal displays are not proof dependencies.

The five-triangle chain is stored under \path{proof/n5/}. Its planar-code inputs \path{proof/n5/p7.pc} and \path{proof/n5/p8.pc}, enumerator, symbolic relation classifier, analytic-bound verifier, and final audit reproduce the 20 and 299 isomorphism-class totals and the complete classification in Table~\ref{tab:n5-classification}. The file \path{proof/n5/plantri_provenance.json} records the graph-input provenance and cross-check procedure.

The seven-triangle artifact is preserved under \path{proof/n7/} in its working layout, including \path{proof/n7/REPRODUCE.md} and \path{proof/n7/MANIFEST.json}. Its final audit reports 14,989 general interval cases, 361,046 boxes, seven interval survivors all handled analytically, and zero unresolved cases. The independent Python and C++ interval verifiers remain separate programs. Their arithmetic uses arbitrary-precision integers and signed 128-bit integers, respectively; the latter has the explicit overflow bound stated in Section~\ref{sec:certification}.

The nine-triangle certificates are stored under \path{proof/n9/}. The program \path{proof/n9/tilted_strip_n9_verifier.py} checks the polynomial and recurrence identities, rational root isolation, geometric inequalities, and nine centered derivative signs for the fixed-topology minimax theorem. Its decimal coordinates are auxiliary output. The program \path{proof/n9/zigzag9_continuous_exclusion.py} gives an exact-rational exclusion of all 256 direction sequences in the complete single-cap two-rail family for \(R\le1/3500\). Neither program claims global optimality of the tilted-strip construction. The optional \path{proof/n9/zigzag13_optional_exclusion.py} supports only the thirteen-triangle remark and is not needed for the main theorems.

The repository-level commands \path{scripts/verify_quick.sh} and \path{scripts/verify_all.sh} provide, respectively, fast checks and the complete proof chain. The reproduction guide, trust model, and proof-to-code map are \path{docs/REPRODUCE.md}, \path{docs/TRUST_MODEL.md}, and \path{docs/PROOF_TO_CODE_MAP.md}. Proof-relevant computations use exact graph algorithms, integer or rational arithmetic, and exact symbolic polynomial operations; cached JSON results, plots, and decimal approximations alone are not certificates.

\paragraph{Code and data availability.}
The source code, executable certificates, machine-readable outputs, and
\LaTeX{} source for this article are available at
\url{https://github.com/limuxi33/monsky_problem_minima}. The exact proof artifact
supporting this version of the manuscript is release
\texttt{v1.0-proof}, archived at
\href{https://doi.org/10.5281/zenodo.22817958}
{\texttt{doi:10.5281/zenodo.22817958}}.
\section{Discussion}

The two exact results illustrate complementary approaches to the quantitative Monsky problem. At \(n=5\), exhaustive topology can be compressed into a small catalogue of algebraic area relations, after which a sharp one-variable inequality solves the problem. At \(n=7\), the decisive new issue is the interaction of many more combinatorial types with T-junctions, including cyclic support dependencies. The five-parameter rational multiaffine representation turns these dependencies from an obstacle into a certifiable interval problem.

The nine-triangle results change the interpretation of the small-\(n\) closure pattern. The common polynomials for \(n=3,5,7\) remain a genuine structural phenomenon, but they do not define a globally competitive geometric family beyond seven triangles. At \(n=9\), the entire single-cap two-rail family is separated from the new upper bound by a substantial gap:
\[
r_{\mathrm{tilt}}<\frac1{3500}.
\]
The better construction is a small geometric deformation of an inherited strip picture rather than a direct continuation of the old cap-plus-zig-zag topology.

This suggests a three-way distinction for future upper-bound searches: new primitive topologies, constructions inherited from smaller seeds, and deformations of inherited constructions. The optimized straight-strip formula already shows that inheritance depends on more than the seed value \(\Delta(k)\), while the tilted-strip construction shows that releasing a geometric degree of freedom can beat the best undeformed extension. The optional \(n=13\) calculation points in the same direction.

The principal next exact problem is \(\Delta(9)\). The tilted-strip value should be regarded as a rigorous benchmark, not as a conjectural global answer. A complete solution requires the four skeleton-node layers \(N=8,9,10,11\), together with a verified parameter model for cyclic and singular T-junction configurations. The computational strategy used for \(n=7\) appears adaptable in spirit, but the seven-parameter geometry must be justified before large-scale interval elimination can be trusted.

A separate structural question is whether a global minimizer can always be chosen with only two distinct triangle areas. The exact cases \(n=3,5,7\) and the fixed-topology \(n=9\) minimax are consistent with this possibility, and centered-gradient arguments explain why endpoint patterns are natural in smooth codimension-one families. At present, however, the possible singular, rank-deficient, and boundary cases are substantial enough that \cref{q:two-level} should remain a question rather than a conjectural theorem-shaped prediction.

Other natural problems include geometric classification of the \(n=7\) minimizers beyond the now-rigid area multiset, a conceptual explanation of the closure polynomials in the exact small cases, and systematic discovery of primitive versus inherited versus deformed record constructions for larger odd \(n\). The trust model remains the same throughout: exploratory floating-point optimization may suggest candidates, but final claims should be reduced to exact symbolic identities, exhaustive graph combinatorics, integer or rational interval certificates, and analytic arguments for a small survivor set.

\section*{Acknowledgements}

The author gratefully acknowledges the substantial contribution of OpenAI's
GPT models, accessed through ChatGPT and Codex, throughout the development of
this project. Through an extended dialogue, GPT helped refine the research
program; explore and improve candidate constructions for the five-, seven-,
and nine-triangle cases; formulate proof decompositions and equality-rigidity
arguments; and contribute to the design, implementation, debugging, and audit
of the combinatorial-enumeration, symbolic-algebra, and exact
interval-verification code. GPT also helped identify gaps and overstatements
in intermediate arguments, organize the reproducibility artifact, and revise
the exposition and repository documentation. Several intermediate claims were
proposed, tested, corrected, or discarded through this collaboration. The
author selected the final mathematical statements, reviewed and corrected the
arguments and code, executed the complete verification chains, and assumes
full responsibility for the results and for any remaining errors.

\printbibliography[heading=bibintoc,title={References}]

@article{Monsky1970,
  author  = {Monsky, Paul},
  title   = {On Dividing a Square Into Triangles},
  journal = {The American Mathematical Monthly},
  volume  = {77},
  number  = {2},
  pages   = {161--164},
  year    = {1970}
}

@article{LabbeRoteZiegler2020,
  author  = {Labb{\'e}, Jean-Philippe and Rote, G{\"u}nter and Ziegler, G{\"u}nter M.},
  title   = {Area Difference Bounds for Dissections of a Square into an Odd Number of Triangles},
  journal = {Experimental Mathematics},
  volume  = {29},
  number  = {3},
  pages   = {253--275},
  year    = {2020},
  doi     = {10.1080/10586458.2018.1459961}
}

@article{CampbellBradyNair2007,
  author  = {Campbell, G. and Brady, J. and Nair, A.},
  title   = {Tiling the Unit Square with 5 Rational Triangles},
  journal = {Rocky Mountain Journal of Mathematics},
  volume  = {37},
  number  = {2},
  pages   = {399--418},
  year    = {2007},
  doi     = {10.1216/rmjm/1181068758}
}

@article{BurtonDattaSpreer2022,
  author  = {Burton, Benjamin A. and Datta, Basudeb and Spreer, Jonathan},
  title   = {Flip Graphs of Stacked and Flag Triangulations of the 2-Sphere},
  journal = {The Electronic Journal of Combinatorics},
  volume  = {29},
  number  = {2},
  pages   = {P2.6},
  year    = {2022},
  doi     = {10.37236/10292}
}

\appendix
\section{Computational certificate map}\label{app:certificates}

Table~\ref{tab:certificate-map} uses paths relative to the repository root. The top-level reproduction guide gives the required order and arguments; some programs consume files generated by earlier stages.

\begingroup
\small
\setlength{\LTleft}{0pt}
\setlength{\LTright}{0pt plus 1fill}
\begin{longtable}{@{}>{\raggedright\arraybackslash}p{0.28\textwidth}>{\raggedright\arraybackslash}p{0.68\textwidth}@{}}
\caption{Map from mathematical claims to executable certificates.}\label{tab:certificate-map}\\
\toprule
Mathematical claim & Repository paths and checks \\
\midrule
\endfirsthead
\toprule
Mathematical claim & Repository paths and checks (continued) \\
\midrule
\endhead
\bottomrule
\endfoot
Five-triangle classification &
\path{proof/n5/enumerate_dissections.py}\newline
\path{proof/n5/classify_relations.py}\newline
20 and 299 classes; complete exact relation table. \\[0.5em]
Five-triangle inequalities and coverage &
\path{proof/n5/verify_analytic_bounds.py}\newline
\path{proof/n5/audit_n5.py}\newline
Exact nonzero bounds, sharp construction, and summary consistency. \\[0.5em]
Fixed seven-triangle zig-zag minimax &
\path{proof/n7/research/fixed_sign_certificate.py}\newline
Strict centered-gradient signs, quartic identity, and root isolation. \\[0.5em]
Other seven-triangle direction sequences &
\path{proof/n7/research/all_signs_certificate.py}\newline
62 sequences excluded for \(R\le1/4000\). \\[0.5em]
Ordinary and seven-node layers &
\path{proof/n7/research/boundary_interval.py}\newline
\path{proof/n7/research/one_interior_interval.py}\newline
\path{proof/n7/research/two_interior_interval.py}\newline
\path{proof/n7/research/single_t_interval.py}\newline
All cases excluded for \(R\le1/1000\). \\[0.5em]
Eight-node auxiliary layers &
\path{proof/n7/research/n8_one_t_interval.py}\newline
\path{proof/n7/research/n8_two_t_verify.py}\newline
Exact interval elimination of the remaining auxiliary layers. \\[0.5em]
Skeleton enumeration and cyclic T-junction model &
\path{proof/n7/research/graph_enumeration.py}\newline
\path{proof/n7/research/general_ratio_cases.py}\newline
Complete canonical type lists and five-variable multiaffine data. \\[0.5em]
Principal interval proof &
\path{proof/n7/research/general_ratio_verify.cpp}\newline
\path{proof/n7/research/ratio_python_verify.py}\newline
Exact integer verification with matching case statuses, box counts, and survivor boxes. \\[0.5em]
Three hard survivors &
\path{proof/n7/research/survivor_hard_certificate.py}\newline
Strict derivative signs and the identity \(-Q_7/2401\). \\[0.5em]
Seven-triangle equality rigidity &
\path{proof/n7/research/fixed_sign_certificate.py}\newline
\path{proof/n7/research/survivor_hard_certificate.py}\newline
The strict signs, together with area-preserving recutting, imply the \(4H_7+3L_7\) multiset as proved in Section~\ref{subsec:n7-rigidity}. \\[0.5em]
Seven-triangle global coverage &
\path{proof/n7/research/global_certificate_audit.py}\newline
Every marked type belongs to a handled category; zero unresolved cases. \\[0.5em]
Tilted-strip nine-triangle minimax &
\path{proof/n9/tilted_strip_n9_verifier.py}\newline
Exact identities, root isolation, geometric inequalities, and derivative signs; auxiliary decimal coordinates. \\[0.5em]
Single-cap nine-triangle exclusion &
\path{proof/n9/zigzag9_continuous_exclusion.py}\newline
Exact-rational exclusion of all 256 directions for \(R\le1/3500\). \\
\end{longtable}
\endgroup

\end{document}